\documentclass[11pt]{article}   % MikTeX
\usepackage[dvips]{graphicx}
\usepackage{amsmath,amssymb}
\usepackage{algorithm}
\usepackage{algorithmic}
\usepackage{lineno,hyperref}
\usepackage[utf8]{inputenc}
\usepackage[english]{babel}
\usepackage{ifthen}

\date{}

\newtheorem{theorem}{Theorem}
\newtheorem{lemma}{Lemma}
\newtheorem{corollary}{Corollary}
\newtheorem{proposition}{Proposition}

\newtheorem{conjecture}{Conjecture}

\date{}

\newenvironment{proof}[1][\hspace{-1.0ex}]%
{\par\addvspace{1mm}{\it Proof\hspace{1.0ex}{#1}.} }%
{\quad$\Box$\par\addvspace{1mm}}

    \newif\ifNoRemark
    \def\addtheorem#1#2#3#4{ % \usepackage{ifthen} needed
    \ifthenelse{\expandafter\isundefined\csname the#2\endcsname}{\newcounter{#2}}{}
    \newenvironment{#1}[1][\global\NoRemarktrue]% No Remark by default
     {\par\addvspace{2mm}\noindent % ????? ???????? ??? ??????? ??????
       \refstepcounter{#2}{\bf #3~\csname the#2\endcsname
      \vphantom{##1}\ifNoRemark.\ \else\ (##1).\fi}\begingroup #4}%
     {\endgroup\par\addvspace{1mm}\global\NoRemarkfalse}
    \expandafter\newcommand\csname b#1\endcsname{\begin{#1}}
    \expandafter\newcommand\csname e#1\endcsname{\end{#1}}
    }

\begin{document}

\title{Every latin hypercube of order $5$ has transversals
\thanks{The research  has been carried out within the framework of a
state assignment of the Ministry of Education and Science of the
Russian Federation for the Institute of Mathematics of the Siberian
Branch of the Russian Academy of Sciences (project no.
FWNF-2022-0017). }}

\author{A. L. Perezhogin$^1$, V. N. Potapov$^1$, S. Yu. Vladimirov$^2$\\
$\quad^1$ Sobolev Institute of Mathematics\\
{pereal@math.nsc.ru, vpotapov@math.nsc.ru}\\
$\quad^2$ Novosibirsk State University\\
{s.vladimirov@g.nsu.ru} }

\maketitle

\begin{abstract}
We prove that for all $n>1$ every latin $n$-dimensional cube of
order $5$ has transversals.  We find all $123$ paratopy classes of
layer-latin cubes of order $5$ with no transversals. For each $n\geq
3$ and  $q\geq 3$ we construct a $(2q-2)$-layer latin
$n$-dimensional cuboid with no transversals.   Moreover, we find all
paratopy classes of nonextendible  and noncompletable latin cuboids
of order $5$.

\end{abstract}

Keywords: transversal, latin square, latin hypercube, permanent of
multidimensional matrix,  nonextendible latin cuboid.

 MSC[2020] 05B15

\section{Introduction}

A {\it latin square} of order $q$ is  a $q\times q$ array of $q$
symbols where each symbol occurs exactly once in every row and in
every column.  An $n$-dimensional array of size $q\times q\times
\cdots\times q$ satisfying the same condition, i.e., every line
contains $q$ different symbols, is called a {\it latin $n$-cube} of
order $q$. Without loss of generality, we will consider latin
squares and cubes under the alphabet $Q_q=\{0,1,\dots,q-1\}$. An
$n$-dimensional array
 $k\times\underbrace{ q\times\dots\times q}_{n-1} $, $k<q$, that
 contains different symbols from $Q_q$ in any line, is called a {\it latin
 $n$-cuboid}. A latin  $(n-1)$-dimensional cube obtained  by
 fixing the first coordinate of this  $n$-cuboid is called  a {\it layer},
  and arrays obtained by fixing any coordinate are called hyperplanes. Layers of  latin
 $n$-cubes may be obtained by fixing an arbitrary coordinate. Layers of a latin
 $n$-cube obtained by fixing the same coordinate are called {\it
 parallel}.

%If $q$ is a prime power then we imply that
%$Q_q$ is the Galois field $GF(q)$.

A {\it row-latin rectangle} of size $ k\times q$ is an array of $q$
symbols, where each row contains all $q$ different symbols.  We
define {\it $k$-layer latin $n$-cuboid} of order $q$ as an
$n$-dimensional generalization of row-latin rectangles such that
each of $k$ layers is a latin $(n-1)$-cube of order $q$. If $k=q$,
then a $k$-layer latin $n$-cuboid of order $q$ is called a {\it
layer-latin $n$-cube} of order $q$. If $n=3$ we will use the terms
cube and cuboid without specifying the dimension.

A  {\it transversal} in a latin square of order $q$ is a set of $q$
cells, where each row contains one cell, each column contains one
cell, and there is one cell containing each symbol. We generalize
the definition of transversals for arbitrary arrays of size
$k_1\times k_2\times \cdots \times k_d$. Let
$k=\min\{k_1,\dots,k_d\}$. A {\it transversal} in the array is a set
of $k$ cells such that every hyperplane (layer) contains at most one
cell of the transversal and all symbols in these cells are pairwise
different. The equivalent definitions of transversals  for row-latin
rectangles and latin $n$-cubes were known early (see \cite{Drisko},
\cite{AAT16}, \cite{Wan}).

In 1967,  Ryser \cite{Ryser} conjectured that, when $q$ is odd,
every latin square of order $q$ has a transversal. In 2011, Wanless
\cite{Wan} generalized Ryser's conjecture to latin hypercubes. He
conjectured that

($1$) if $q$ is odd, then for any $n\geq 2$ every latin $n$-cube of
order $q$ has a transversal;

($2$) if $q$ is even, then for any odd $n\geq 3$ every latin
$n$-cube of order $q$ has a transversal.

In \cite{AAT18} Taranenko proved Wanless's conjecture  for $q=4$. In
the present paper we prove Wanless's conjecture  for $q=5$ (Theorem
\ref{PPVth2}). Our result is based on the list of all paratopy
classes of $2$-layer latin cuboids of order $5$. This list is
constructed by  computer-aided calculations (see Section 5). It is
available in the internet data archive
https://zenodo.org/records/10204026 .

 A latin cuboid $C$ of size $k\times
q\times q$, $k<q$, is called {\it extendible} ({\it completable}) if
there exists a latin cuboid $(k+1)\times q\times q$ (a latin cube of
order $q$) that contains $C$. With the help of the  list of all
paratopy classes of $2$-layer latin cuboids of order $5$, we find
all paratopy classes of nonextendible and noncompletable latin
cuboids of order $5$. More results on nonextendible and
noncompletable latin cuboids are available in  \cite{BCMPW} and
\cite{McKW}.

In 1998,  Drisko \cite{Drisko} proved that any {row-latin rectangle}
of size $k\times q$, $k\geq 2q-1$, has a transversal, and that for
any $q>1$ and there exists a {row-latin rectangle} of size
$(2q-2)\times q$ with no transversals. In this paper we prove that
for any  $q\geq 3$ there exists a $(2q-2)$-layer latin $n$-cuboid of
order $q$ with no transversals (Theorem \ref{PPVth1}).

A nonnegative multidimensional matrix  is called  {\it
$t$-stochastic} if it has the same sums of elements in each
$t$-dimensional plane. A $1$-stochastic matrix is called
polystochastic. Latin $n$-cubes one-to-one correspond to the Cayley
tables of $n$-ary quasigroup operations, and graphs of $n$-ary
quasigroup operations one-to-one correspond to polystochastic
$(n+1)$-dimensional $(0,1)$-matrices (see \cite{AAT16}).  Any
layer-latin $n$-cube corresponds to a $2$-stochastic
$(n+1)$-dimensional $(0,1)$-matrix in the same way. It follows
straightforwardly from the definitions (see Theorem \ref{PPVth5}).
However, this correspondence is only injective. Transversals of a
layer-latin $n$-cube correspond to unity diagonals of
$(n+1)$-dimensional matrix. The {\it permanent} of a
multidimensional matrix is the sum of the products of entries over
all diagonals. Thus, the number of unity diagonals  in any
$(0,1)$-matrix is equal to the permanent of the matrix. In
\cite{AAT16} Taranenko generalized Wanless's conjecture in the
following way:

($1'$) the permanent of every polystochastic matrix of odd order
$q$ is positive;

($2'$) the permanent of every polystochastic $n$-dimensional matrix
of even order $q$ is positive if $n$ is even.

Taranenko proved her conjecture in two cases: for $q=3$ and an
arbitrary $n$ \cite{AAT16},  for $q=4$ and $n= 4$ \cite{AAT20}.
  We
construct $2$-stochastic $n$-dimensional matrices of order $q$ with
zero permanent for all $n>3$ and $q\geq 2$ (Theorem \ref{PPVth5}).
Consequently, it is not possible to generalize Taranenko's
conjecture to $2$-stochastic matrices.

In 1975, Stein \cite{Stein} conjectured that any latin square of
even order $q$ has a partial transversal of size $q-1$.  Moreover,
he propagated this conjecture for arrays with weaker  conditions,
namely for  row-latin squares. Stein's conjecture is also known from
\cite{BR}. Recently, Montgomery \cite{Mont} proved an asymptotical
(for sufficiently large order) version of Stein's hypothesis.
 A previous estimation on the length of partial transversals in latin
squares is proved in \cite{KPSY}. In \cite{PS} it is proved that an
analog of Stein's conjecture is false for arrays of size $q\times q$
containing every symbol $q$ times. Recently, it was shown that
almost all latin squares (for the growing order) have a lot of
transversals \cite{EMM, Kwan}. In \cite{BPW},  \cite{AAT18'},  and
\cite{AAT20} other generalizations of transversals in latin squares
and hypercubes are considered.

\section{Main Results}

It is well known that latin $n$-cubes of order $q$ are the same as
graphs of $n$-ary quasigroups of order $q$. More precisely,  any
latin $n$-cube of order $q$ is an array $\{(x_0,x_1,\dots,x_n) :
x_0=f(x_1,\dots,x_n)\}$ where $f:Q_q^n\rightarrow Q_q$ is an $n$-ary
quasigroup operation (or $n$-ary quasigroup for brevity). Let
$(\sigma_0,\sigma_1,\dots,\sigma_n)$ be an ordered set of
permutations of $Q_q$. It is easy to see that if $f$ is an $n$-ary
quasigroup, then
$g(x_1,\dots,x_n)=\sigma_0^{-1}(f(\sigma_1(x_1),\dots,\sigma_n(x_n)))$
is also an $n$-ary quasigroup. Such $n$-ary  quasigroups $g$ and $f$
(and the corresponding  latin $n$-cubes) are called {\it isotopic}.
The set of all latin $n$-cubes isotopic to a fixed latin $n$-cube is
called an {\it isotopy class}.
 Let $\tau$ be a permutation  of coordinates, i.e.,
 $\tau:\{0,1,\dots,n\}\rightarrow \{0,1,\dots,n\}$ is a bijection.
 If $$\{(x_0,x_1,\dots,x_n) :
x_0=f(x_1,\dots,x_n)\}=
\{(x_{\tau(0)},x_{\tau(1)},\dots,x_{\tau(n)}) :
x_0=g(x_1,\dots,x_n)\},$$ then $n$-ary quasigroups $f$ and $g$ are
called {\it parastrophic}. The equivalence class of all parastrophic
and isotopic $n$-ary quasigroups is called  a {\it paratopy class}.
If we are limited to permutations with  $\tau(0)=0$, then we will
use the term {\it incomplete paratopy class}. Two latin $n$-cubes
 are called {\it equivalent} if the corresponding
$n$-quasigroups belong to the same paratopy class. These definitions
are generalized to $k$-layer latin $n$-cuboids straightforwardly.

The number of latin cubes of order $5$ was counted in \cite{MW}, and
the number of paratopy classes of such cubes can be found in
\cite[Table 1]{McKW}.

\begin{proposition}\label{PPVprop1}
There are $2$ paratopy classes of latin squares of order $5$ and
$15$ paratopy classes of latin cubes of order $5$.
\end{proposition}

A latin $n$-cube is called {\it linear} if it is equivalent to the
$n$-cube defined by the equation  $x_0= x_1+x_2+\dots+x_n$. Here and
below, we use $+$ to designate the addition by modulo $q$. Note that
only one paratopy class of latin squares of order $5$ consists of
nonlinear latin squares.

It is easy to verify the following statements:

\begin{corollary}\label{PPVcor1} \quad

$(1)$ Every latin cube of order $5$ has a transversal.

$(2)$ If a latin cube of order $5$ has a nonlinear layer, then it
has another parallel nonlinear layer.
\end{corollary}

An $n$-ary quasigroup $f$ is called {\it reducible} if $f$ is a
composition of a $k$-ary quasigroup $h$ and  an $(n-k+1)$-ary
quasigroup $g$, where $k\geq 2$ and $n-k+1\geq 2$, i.e.,
$$f(x_1,\dots,x_k,x_{k+1},\dots,x_n)=g(h(x_1,\dots,x_k),x_{k+1},\dots,x_n).$$
  An $n$-ary quasigroup $f$ is called {\it permutably reducible} if
it is equivalent to some reducible $n$-ary quasigroup.  The
corresponding to $f$ latin $n$-cube is also called permutably
reducible. It is easy to see that a permutably reducible latin
$n$-cube can be obtained by a composition of two latin cubes with
fewer dimensions.

A latin $(n-k)$-cube obtained by by fixing some $k$ coordinates in
the latin $n$-cube  $C$ is called a {\it retract} of $C$. Parallel
retracts obtained by fixing the first $k$ coordinates in $C$ with
values $z_1,\dots,z_j,\dots, z_k$ and $z_1,\dots,z'_j,\dots, z_k$,
$z_i\in Q_q$, $z_j\neq z'_j$, are called adjacent in $j$th
direction. These adjacent retracts are layers of the retract of $C$
obtained by fixing  $k-1$ coordinates (the first $k$ coordinates
with the exception of $j$th coordinate)  with values
$z_1,\dots,z_{j-1},z_{j+1},\dots, z_k$. By fixing all coordinates
except for two coordinates of a latin $n$-cube, we obtain a latin
square. It is called a {\it square retract} of $n$-cube. A latin
$n$-cube is called {\it sublinear}, if all its square retracts
 are linear.

\begin{proposition}[\cite{KP}, Corollary 2]\label{PPV-KP}
All sublinear latin $n$-cubes of order $5$ are permutably reducible
for $n \geq 4$.
\end{proposition}

\begin{proposition}[\cite{AAT16}, Proposition 13]\label{PPV-AAT}
Let a latin $n$-cube $f$ be permutably reducible, i.e., $f$ is
equivalent to the composition of $g$ and $h$. If $g$ and $h$ have
transversals, then $f$ has a transversal.
\end{proposition}

\begin{corollary}\label{latinCor1}
Every sublinear latin $n$-cube of order $5$ contains transversals.
\end{corollary}
\begin{proof}
Proof is by induction. For $n=2,3$ this statement is true by the
complete search.  Suppose that we prove the corollary for $k\leq n$.
Consider a sublinear latin $(n+1)$-cube $L$ of order $5$. By
Proposition \ref{PPV-KP}, $L$ is a composition of two latin cubes
with fewer dimensions. By definition, these two latin cubes are
sublinear. By induction, these two latin cubes contain transversals.
By Proposition \ref{PPV-AAT}, $L$ contains  transversals.
\end{proof}

 Let $(y^1_1,\dots,y^1_k),\dots,(y^q_1,\dots,y^q_k)$ be
a diagonal, i.e., for every $j=1,\dots,k$ elements $y^1_j,\dots,
y^q_j$ are pairwise different. Consider  {\it retract} $C_i$,
$i=1,\dots,q$, of a latin or layer-latin $n$-cube $C$ obtained  by
fixing the first $k$ coordinates in $C$ with values $y^i_1,\dots,
y^i_k$.
 A
layer-latin $(n-k+1)$-cube consisting of layers $C_1,\dots,C_q$ is
called {\it diagonal latin-layer retract} of $C$. The following
proposition is true by the definition of a transversal.

\begin{proposition}\label{PPVprop4}
%If a latin-layer or latin hypercube contains a  diagonal latin-layer
%subcube, then  any transversal of the subcube is a transversal of the
%hypercube.
Any transversal  of the diagonal latin-layer retract is a
transversal of the hypercube.
\end{proposition}

Through computer-aided calculations, we  establish

\begin{theorem}\label{latinT1}
Every layer-latin cube of order $5$ with no transversals has no more
than one nonlinear layer.
\end{theorem}
We will explain these calculations in detail in Section 5.

\begin{proposition}\label{latinProp1}
If a latin hypercube of order $5$ contains a nonlinear  square
retract, then it has  transversals.
\end{proposition}
\begin{proof} Let $C$ be a latin $n$-cube of order $5$ satisfying
conditions of the proposition. Without loss of generality, we assume
that a nonlinear square retract of $C$ is obtained by fixing all
coordinates except for $n$th and $(n-1)$th coordinates. By Corollary
\ref{PPVcor1} (2), if $C$ contains a nonlinear square retract, then
$C$ contains other
 nonlinear square retracts these are adjacent the first one in any direction.
 Thus we can
find a sequence of $n-1$  nonlinear square retract such that the
$1$st and the $2$nd square retracts are adjacent in the first
direction, the $2$nd and the $3$rd square retracts are adjacent in
the second direction, and so on. Take the first and last square
retracts from this sequence. By the construction, these two square
retracts were obtained from $C$ by fixing $1$st, $2$nd, ...,
$(n-2)$th coordinates with different values. Consider a diagonal
layer-latin retract that contains both of these square retracts as
layers. By Theorem \ref{latinT1}, this diagonal layer-latin square
retract has a transversal. By Proposition \ref{PPVprop4}, $C$ has
transversals.
\end{proof}

\begin{theorem}\label{PPVth2}
Every latin $n$-cube of order $5$ has transversals.
\end{theorem}
\begin{proof} If $n=3$, then a latin $3$-cube of order $5$ has transversals
by Corollary \ref{PPVcor1} (1). Let $n\geq 4$. If a latin $n$-cube
contains a nonlinear square retract, then it has transversals by
Proposition \ref{latinProp1}. If it does not  a nonlinear square
retract, then the latin $n$-cube is sublinear, so it has
transversals by Corollary \ref{latinCor1}. \end{proof}

\section{Layer-latin cubes}

Let $L[\pi,n]$ be a linear latin $n$-cube that is defined by the
equation:
$$L[\pi,n]=\pi x_1+x_2+\dots+x_n.$$

\begin{lemma}\label{PPVlemma1}\quad

$(1)$ A layer-latin $(n+1)$-cube of odd  order $q$ consisting of
layers
$L[\pi_1,n],\dots, L[\pi_q,n]$ does not have transversals, if %and only if
the row-latin square consisting of layers $\pi_1,\dots,\pi_q$ does
not  have a diagonal with zero sum.

$(2)$ A layer-latin $(n+1)$-cube of even order $q$ consisting of
layers
$L[\pi_1,n],\dots, L[\pi_q,n]$ does not have transversals, if %and only if
the row-latin square consisting of layers $\pi_1,\dots,\pi_q$ does
not have a diagonal with a zero sum if $n$ is even and with $q/2$
sum if $n$ is odd.
\end{lemma}
\begin{proof} %($\Rightarrow$)
(1) A transversal
$\{(y^1_0,y^1_1,\dots,y^1_n),(y^2_0,y^2_1,\dots,y^2_n),
\dots,(y^q_0,y^q_1,\dots,y^q_n)\}$ of $L[\pi_1,n]$,\dots,
$L[\pi_q,n]$
 satisfy the following system of linear equations:

$$
 \left\{
\begin{array}{l}
y_0^1=\pi_1 y_1^1+y_2^1+\dots+y^1_n,\\
y_0^2=\pi_2 y_1^2+y_2^2+\dots+y^2_n,\\
\dots\\
y_0^q=\pi_q y_1^q+y_2^q+\dots+y^q_n.\\
\end{array}
\right.
$$

By the definition of transversals, for each $i=0,1,\dots,n$  sets
$\{y_i^1,\dots,y_i^q\}$ coincide with $Q_q$. The sum of all elements
from $Q$ is equal to zero by modulo $q$. Consequently,
$\pi_1y_1^1+\dots+\pi_qy_1^q=0$, i.e., the row-latin square
consisting of layers $\pi_1,\dots,\pi_q$ has a diagonal with zero
sum.

The proof of item (2) is analogous to item (1), with the exception
that the sum of all elements from $Q$ is equal to $q/2$ by modulo
$q$ if $q$ is even.
%($\Rightarrow$)
\end{proof}

\begin{corollary}\label{corol_pi}
A layer-latin cube of odd order $5$ consisting of layers
$L[\pi_1,2],\dots, L[\pi_5,2]$ does not have transversals if %and only if
the row-latin square consisting of layers $\pi_1,\dots,\pi_5$ does
not  have a diagonal with a zero sum.
\end{corollary}

Denote by $\varepsilon_i$ a permutation of $q$ elements such that
$\varepsilon_ix=x+i$.

\begin{corollary}\label{corol_pi0}\quad

$(1)$ A layer-latin $(n+1)$-cube of odd  order $q$ consisting of
layers $L[\varepsilon_{i_0},n],L[\varepsilon_{i_1},n]$,\dots,
$L[\varepsilon_{i_{q-1}},n]$ has transversals if and only if
$\sum\limits_{j=0}^{q-1}i_j=0$.

$(2)$ A layer-latin $(n+1)$-cube of even order $q$ consisting of
layers $L[\varepsilon_{i_0},n],L[\varepsilon_{i_1},n]$,\dots,
$L[\varepsilon_{i_{q-1}},n]$ has transversals if and only if
$\sum\limits_{j=0}^{q-1}i_j=0$ if $n$ is odd and
$\sum\limits_{j=0}^{q-1}i_j=q/2$ if $n$ is even.
\end{corollary}

\begin{proposition}\label{PPVprop22}
Every cell of a layer-latin $(n+1)$-cube $C$ of  order $q$
consisting of layers $L[\varepsilon_{i_0},n]$,
$L[\varepsilon_{i_1},n]$, \dots, $L[\varepsilon_{i_{q-1}},n]$ is
included in the same number of transversals.
\end{proposition}
\begin{proof}  Consider a  linear latin $n$-cube
$L[\varepsilon_{0},n]$ defined by the equation $x_0=x_1+\dots+x_n$.
Let tuples $(y_0,y_1,\dots,y_n)$ and $(z_0,z_1,\dots,z_n)$ suffice
the equation. Consider an isotopy $\tau_0=\sum_{i=1}^n z_i-y_i$,
$\tau_i=z_i-y_i$, $i=1,\dots,n$. It maps $(y_0,y_1,\dots,y_n)$ to
$(z_0,z_1,\dots,z_n)$  and preserves $L[\varepsilon_{0},n]$. It is
easy to see that this map is an autotopy of $L[\varepsilon_{j},n]$
for any $j\in\{1,\dots,q-1\}$. Therefore,
$\tau=(\tau_0,\dots,\tau_n)$ is an autotopy of $C$. Each autotopy
maps  transversals to transversals. Then cells $(y_1,\dots,y_n)$ and
$(z_1,\dots,z_n)$ in any fixed layer are included in the same number
of transversals of $C$. The number of transversals is equal for all
cells from any fixed layer  because $(y_1,\dots,y_n)$ and
$(z_1,\dots,z_n)$ are arbitrary.
 By the definition of transversal,
every transversal intersects each layer exactly once. Then the
number of transversals is equal for any cells from different layers.
\end{proof}

\begin{lemma}\label{PPVlemma2}
Let $\sum\limits_{j=0}^{q-1}i_j=0$. A layer-latin $(n+1)$-cube of
odd order $q$ consisting of layers
$L_0,L[\varepsilon_{i_1},n],\dots, L[\varepsilon_{i_{q-1}},n]$ has
no transversals if and only if
 $n$-cubes $L_0$ and
$L[\varepsilon_{i_0},n]$ are different in each cell.
\end{lemma}
\begin{proof} Suppose that this layer-latin $(n+1)$-cube has a transversal.
Consider a cell $z = (z_1,\dots,z_n)$ with value
$L_0(z_1,\dots,z_n)$ from the transversal. By the conditions of the
lemma, $L_0(z)=L[\varepsilon_{j},n](z)$ for some $j\neq i_0$.  Then
the layer-latin $(n+1)$-cube  consisting of layers
$L[\varepsilon_{j},n],L[\varepsilon_{i_1},n],\dots,
L[\varepsilon_{i_{q-1}},n]$ has a transversal under the same set of
cells. But
 by
Corollary \ref{corol_pi0}, for any $j\neq i_0$ the latin-layer
$(n+1)$-cube consisting of layers
$L[\varepsilon_{j},n],L[\varepsilon_{i_1},n],\dots,
L[\varepsilon_{i_{q-1}},n]$  does not have  transversals.

The inverse statement follows from Proposition \ref{PPVprop22}. If
$L_0(z)=L[\varepsilon_{i_0},n](z)$,  then there exists a transversal
containing cell $z$. \end{proof}

\begin{theorem}\label{PPVth1}
 For any  $q>1$ there exists a $(2q-2)$-layer latin
$(n+1)$-cuboid of order $q$ that does not have transversals.
\end{theorem}
\begin{proof}  Case 1: $q$ is odd or $q$ is even and $n$ is odd.
 Consider a layer-latin $(n+1)$-cuboid consisting of $q-1$ layers
 $L[\varepsilon_{0},n]$ and $q-1$ layers
 $L[\varepsilon_{1},n]$.
 This layer-latin
$(n+1)$-cuboid does not have transversal
  by  Corollary \ref{corol_pi0}(1,2).

Case 2:  $q$ is even and $n$ is even.  Consider a layer-latin
$(n+1)$-cuboid $C$  consisting of $m\geq q$ identical layers
$L[\varepsilon_{0},n]$. Any transversal of $C$ is equivalent to a
transversal of $L[\varepsilon_{0},n]$. But latin $n$-cube
$L[\varepsilon_{0},n]$ does not have transversals by Corollary
\ref{corol_pi0}(2).
  \end{proof}

The previous statements of this section are well known in the case
of row-latin rectangles (see \cite{Drisko}).  We assume that it is
possible to generalize the main results of \cite{Drisko} as well.

\begin{conjecture}
 Every $(2q-1)$-layer latin $n$-cuboid of odd order $q$ has
transversals.
\end{conjecture}

\begin{theorem}\label{PPVth5}
There exist $2$-stochastic $(n+2)$-dimensional matrices of order $q$
with zero permanent for all $n>1$ and odd $q$.
\end{theorem}
\begin{proof} Consider a layer-latin $(n+1)$-cube $C$ of order $q$, where
$i$-th layer is defined by equation $z_i=f_i(y_1,y_2,\dots,y_n)$,
$i=0,\dots,q-1$. Let us define an $(n+2)$-dimensional matrix $M$ of
size $\underbrace{q\times\dots\times q}_{n+2}$ by the following
rule:
$$
M(i,y_0,y_1,\dots,y_n)= \left\{
\begin{array}{l}
1,\ {\rm if}\ y_0=f_i(y_1,y_2,\dots,y_n),\\
0,\ {\rm otherwise.}\\
\end{array}
\right. $$

Consider a line determined by a nonfixed coordinate $y_j$,
$j=0,\dots, q-1$. Then we have fixed value  $i$ in the first
coordinate.  By the definition of latin $n$-cubes, the equation
$y_0=f_i(y_1,y_2,\dots,y_n)$ has a unique solution. Then such a line
of
 $M$ contains only one unit. Therefore, $M$ has $q$ units in any
$2$-dimensional face. Then matrix $M$ is $2$-stochastic by
definition.

Let $C$ consist of $q$ layers taken from $(2q-2)$-layer latin
$(n+1)$-cuboid of order $q$ with no transversals. Then $C$ does not
have transversals. It is easy to see that transversals of $C$
one-to-one correspond to unity diagonals of $M$. Since ${\rm per} M$
is equal to the number of unity diagonals consisting, the proof is
completed by Theorem \ref{PPVth1}. \end{proof}

\section{Listing of layer-latin cubes of order $5$ with no transversals. Preliminaries}

In this section we prove some properties of $2$-layer ($3$-layer)
latin cuboids of order $5$. We need these properties to search
effectively  layer-latin cubes  with no transversals.

 Define a latin rectangle
of size $k_1\times k_2$ under the alphabet $Q_q$ as a table
$k_1\times k_2$ consisting of elements from $Q_q$ such that elements
of every row and every column are distinct. We can get a latin
rectangle of size $k_1\times k_2$ under alphabet $Q_q$ from latin
squares of order $q$ by choosing $k_1$ rows and $k_2$ columns.

We will call the set of latin rectangles a {\it layer-latin
parallelepiped}. We consider layer-latin parallelepipeds of size
$2\times 5\times 2$ ($3\times 5\times 3$) chosen from $2$-layer
($3$-layer) latin cuboids of order $5$. For brevity, we will call
the $2\times 2\times 2$ ($3\times 3\times 3$) subcube of the
$2$-layer ($3$-layer) latin cuboid a $2$-cube ($3$-cube). Two
$2$-cubes derived from the same layer-latin parallelepiped are
adjacent if they have a nonempty intersection. A diagonal of
$2$-cube ($3$-cube) is a set of $2$ ($3$) cells that intersect all
hyperplains. Recall that we  use the term transversal for diagonals
that contain pairwise different values. The set of cell values is a
{\it range} of transversal.
 Transversals of $2$- or $3$-cubes
can have $10={5 \choose 2}={5 \choose 3}$ different ranges. A
$2$-cube has $4$ diagonals and it can have no more than $4$
different ranges of transversals.

Examples.

1. A $2$-cube with layers
 $\begin{array}{cc}
2&0 \\
0&1
\end{array}$,  $\begin{array}{cc}
1&2 \\
3&0
\end{array}$ has transversals with  $2$ different ranges
 $\{0,3\}$ and $\{0,2\}$.

2.  A $3$-cube with layers
 $\begin{array}{ccc}
\bf{0}&1&2 \\
1&2& 3\\
\it{2}&3 &4
\end{array}$,  $\begin{array}{ccc}
1&2&3 \\
2&\bf{3}& \it{4}\\
3&4&0
\end{array}$,
$\begin{array}{ccc}
2&\it{3}&4 \\
3&4& 0\\
4&0&\bf{1}
\end{array}$
 has transversals with $2$ different ranges
$\{0,3,1\}$ (indicate by bold font) and $\{2,3,4\}$ (indicate by
italic).

For every $2$-layer (or $3$-layer) latin cuboid $A$ of order $5$, we
define a table $T(A)$ of size $10\times 10$. Let us numerate all
$10$ subsets of $Q_5$ with cardinality $2$.  Any $3$-element subset
of $Q_5$ is the complement of some $2$-element subset. Let us label
a $2$-element set $S$ by the same number $i$,  $i\in
\{1,\dots,10\}$, as its $3$-element complement $Q_5\setminus S$.
Fixing the $i$th subset of rows and the $j$th subset of columns of
$A$, we get a $2$-cube (or  a $3$-cube) of $A$.  The number of
different ranges of transversals in the corresponding $2$-cube (or
$3$-cube) of $A$ is the element $T(A)(i,j)$ of $T(A)$. We will say
that labels $i,j\in \{1,\dots,10\}$ are adjacent if the
corresponding $2$-cubes are adjacent.

Example. Consider a $2$-layer  latin cuboid
$$A=\quad
   \begin{array}{ccccccccccc}
1 &|2 &0| &3 &4 &\quad & 1 &|0 &4| &2 &3 \\
0 &|1 &3| &4 &2 &\quad & 0 &|2 &1| &3 &4 \\
\hline
2 &|\bf{3} &\bf{4}| &0 &1 &\quad & 3 &|\bf{4} &\bf{0}| &1 &2 \\
3 &|\bf{4} &\bf{2}| &1 &0 &\quad & 4 &|\bf{3} &\bf{2}| &0 &1 \\
\hline 4 &|0 &1| &2 &3 &\quad & 2 &|1 &3| &4 &0
   \end{array}.
$$
Consider  $2$-element subsets of the $5$-element set with
characteristic functions $(01100)$ and $(00110)$. The first subset
has label $5$ in the lexicographic order and the second subset has
label $8$. The  $2$-cube  corresponding to the  subset $(00110)$ of
rows and the subset  $(01100)$ of columns consists of layers
$\begin{array}{cc}
3&4 \\
4&2
\end{array}$ and  $\begin{array}{cc}
4&0 \\
3&2
\end{array}$. This $2$-cube has $\{3,2\}$, $\{4,3\}$, $\{4,0\}$ and $\{2,4\}$ ranges
of transversals. Then $T(A)(8,5)=4$.

The following proposition can be derived from the pigeonhole
principle.

\begin{proposition}\label{PPVprop5}
Let $A$ be two layers of a layer-latin cube $X$ of order $5$ and let
$A'$ be the other three layers of $X$. If there exist $i,j\in
\{1,\dots,10\}$ such that $T(A)(i,j)+T(A')(i,j)\geq 11$, then $X$
contains a transversal.
\end{proposition}

Now we will study some properties of  ranges of transversals in
$2$-cubes.
 Any $2$-cube has two
layers, where every layer is a latin subsquare. Every subsquare
consists of two columns. Let $\alpha$ be a column $\begin{array}{c}
a \\
b
\end{array}$. By $\overline{\alpha}$  we designate a column
$\begin{array}{c}
b \\
a
\end{array}$.

The following properties of $2$-cubes can be verified
straightforwardly.

\begin{proposition}\label{PPVprop2cube}

Suppose that a $2$-cube $F$ is contained in a layer-latin
parallelepiped of size $2\times 5\times 2$. Then $F$ satisfies the
following properties:

$(1)$ If $F$ does not have transversals and it has the first layer
$\alpha\beta$, then the second layer of $F$ is
$\overline{\beta}\overline{\alpha}$.

$(2)$ If $F$ does not have  transversals, then any adjacent $2$-cube
has $4$ transversals.

$(3)$ If $F$ has $4$ transversals with the same range, then $F$
consists of layers of types $\alpha\overline{\alpha}$ and
$\overline{\alpha}\alpha$.

$(4)$ If $F$ has $4$ transversals and contains $5$ different
symbols, then $F$ has transversals of at least $3$ ranges.

$(5)$ If $F$ has $4$ transversals with two different ranges and the
second columns of  layers of $F$ have types $\beta$ and
$\overline{\beta}$, then $F$ consists of layers of types
$\alpha\beta$ and $\overline{\alpha}\overline{\beta}$ or $F$ is
constituted by two layers, where the first layer is
 $\begin{array}{cc}
b&a \\
a&b
\end{array}$   and another layer  is
 $\begin{array}{cc}
a&b \\
c&a
\end{array}$ or $\begin{array}{cc}
c&b \\
b&a
\end{array}$.

\end{proposition}

\begin{proposition}\label{PPVprop2cube2}
Suppose that a $2$-cube is chosen from a layer-latin parallelepiped
of size $2\times 5\times 2$. If this $2$-cube does not have
transversals, then any adjacent $2$-cube has transversals of  at
least $2$ different ranges.
\end{proposition}
\begin{proof} By Proposition \ref{PPVprop2cube} (1), we obtain that the
first layer of the parallelepiped has columns $\alpha\beta\gamma$,
and the second layer of the parallelepiped has columns
$\overline{\beta}\overline{\alpha}\delta$.  By Proposition
\ref{PPVprop2cube} (2),  an adjacent $2$-cube has a unique range of
transversals.  Hence, it holds $\overline{\gamma}=\delta$. Moreover,
$\gamma=\overline{\alpha}$ and $\delta=\beta=\alpha$ or
$\gamma=\overline{\beta}$ and $\delta=\alpha=\beta$ by Proposition
\ref{PPVprop2cube} (3). Two identical columns in the layer
contradict the definition of layer-latin parallelepiped. \end{proof}

By Propositions \ref{PPVprop5} and \ref{PPVprop2cube2}, we have

\begin{corollary}\label{PPVprop2cube21}
Let $X$ be a layer-latin cube of order $5$,  and let $A'$ be some
three layers of $X$. If $T(A')(i,j)+T(A')(i,k)\geq 19$ or
$T(A')(j,i)+T(A')(k,i)\geq 19$ for some $i$ and some adjacent $j$
and $k$, then $X$ has transversal.
\end{corollary}

\begin{proposition}\label{PPVprop2cube3}
If  a layer-latin parallelepiped of size $2\times 5\times 2$
contains two nonadjacent $2$-cubes with no transversals, then any
other $2$-cube from the  parallelepiped has transversals of  at
least $3$ different ranges.
\end{proposition}
\begin{proof} All lines of length $5$ of the layer-latin parallelepiped
contains all elements from $Q_5$. Consequently, the layers of the
parallelepiped has types  $\alpha\beta\gamma\delta\epsilon$ and
$\overline{\beta}\overline{\alpha}\overline{\delta}\overline{\gamma}\overline{\epsilon}$
by Proposition \ref{PPVprop2cube} (1). If columns $\alpha, \beta,
\gamma, \delta$ contain only $4$ symbols, then $\epsilon$ consists
of the fifth symbol. It contradicts the definition of a layer-latin
parallelepiped. Consider $2$-cubes avoiding column $\epsilon$.
Without loss of generality, take a $2$-cube with layers
$\alpha\gamma$ and $\overline{\beta}\overline{\delta}$. This
$2$-cube contains $5$ symbols, so it has transversals of at least
$3$ different ranges by Proposition \ref{PPVprop2cube}(4). Consider
$2$-cubes containing column $\epsilon$. Without loss of generality,
take a $2$-cube with layers $\alpha\epsilon$ and
$\overline{\beta}\overline{\epsilon}$. If this $2$-cube contains
transversals of $1$ or $2$ different ranges, then $\alpha=\beta$ or
$\alpha=\begin{array}{c}
b \\
a
\end{array}$, $\overline{\beta}=\begin{array}{c}
a \\
c
\end{array}$ or $\overline{\beta}=\begin{array}{c}
c \\
b
\end{array}$
by Proposition \ref{PPVprop2cube} (3,5). All these cases contradict
the definition of a layer-latin parallelepiped, because columns
$\alpha$ and $\beta$ contain equal symbols in the first or second
positions. Then every $2$-cube with column $\epsilon$ has
transversals of  at least $3$ different ranges. \end{proof}

By Propositions \ref{PPVprop5} and \ref{PPVprop2cube3}, we have

\begin{corollary}\label{PPVprop2cube22}
Let $X$ be a layer-latin $3$-cube of order $5$,  and let $A'$ be
some three layers of $X$. If any row $T(A')(i,*)$ or any column
$T(A')(*,i)$ of $T(A)$ contains  elements $10, 10$,  and $y$, where
$y\geq 8$,  then $X$ has a transversal.
\end{corollary}

\section{Listing of layer-latin cubes of order 5 with no transversals. Algorithms}

By the definitions of isotopy and parastrophy, we obtain the
following proposition.

\begin{proposition}\label{PPVprop00}
The number of isotopies acting on a $k$-layer $n$-cuboid of order
$q$ equals\\ $q!((q-1)!)^{n-1}k!$. The number of parastrophies
acting on a $k$-layer $n$-cuboid of order $q$ equals $n!$.
\end{proposition}

\begin{corollary}
Cardinalities of  paratopy classes of $2$-layer latin cuboids are
not greater than $5!\cdot (4!)^2\cdot 2\cdot 3!$, cardinalities of
incomplete paratopy classes of $2$-layer latin cuboids are not
greater than $N_0=5!\cdot (4!)^2\cdot 4$.
\end{corollary}

\begin{proposition}[\cite{McKW}, Table 1]\label{PPVprop111}
There are $N_1=56\cdot 5!\cdot 4!$ different latin cubes of order
$5$.
\end{proposition}

By complete search, we find that

\begin{proposition}\label{PPVprop112}
There are $1625$ paratopy classes and $N_2=4427$ incomplete paratopy
classes of $2$-layer latin cuboids of order $5$.
\end{proposition}

The list of  representatives of paratopy classes is in the internet
data archive https://zenodo.org/records/10204026 .

We implement two algorithms for finding layer-latin cubes of order
$5$ with no transversals.

\subsection{Algorithm 1}

{\bf for} $i=1$ {\bf to} $i=N_2$

\hskip 10mm Take the $i$th $2$-layer latin cuboid with layers $A_0,
A_1$

\hskip 10mm {\bf for} $j=1$ {\bf to} $j=N_1$

\hskip 15mm Take  the $j$th  latin square  $B_2$.

\hskip 15mm Consider the $3$-layer latin cuboid $A'$ with layers
$A_0, A_1, B_2$.

\hskip 15mm {\bf If} $T(A')$ satisfies the conditions of Corollary
\ref{PPVprop2cube21} or \ref{PPVprop2cube22}

\hskip 20mm {\bf Then} $j:=j+1$

\hskip 20mm {\bf Else for} $k=1$ {\bf to} $k=N_1$

\hskip 25mm Take the $k$th  latin square  $B_3$.

\hskip 25mm  Construct cell by cell a latin square $B_4$  such that

\hskip 25mm  the layer-latin cube with layers $A_0, A_1, B_2, B_3,
B_4$

\hskip 25mm does not contain transversals, $k:=k+1$.

\hskip 20mm $j:=j+1$

\hskip 10mm  $i:=i+1$

\subsection{Algorithm 2}

{\bf for} $i=1$ {\bf to} $i=N_2$

\hskip 10mm Take the $i$th $2$-layer latin cuboid with layers $A_0,
A_1$

\hskip 10mm {\bf for} $j=1$ {\bf to} $j=N_1$

\hskip 15mm Take the $j$th  latin square  $B_2$.

\hskip 15mm Consider the $3$-layer latin cuboid $A'$ with layers
$A_0, A_1, B_2$.

\hskip 15mm {\bf If} $T(A')$ satisfies the conditions of Corollary
\ref{PPVprop2cube21} or \ref{PPVprop2cube22}

\hskip 20mm {\bf Then} $j:=j+1$

\hskip 20mm {\bf Else for} $k=1$ {\bf to} $k=N_2$

\hskip 25mm Take  the $k$th   $2$-layer latin cuboid with layers
$B_3, B_4$.

\hskip 25mm {\bf  for} $m=1$ {\bf to} $m=N_0$

\hskip 30mm  Get $B'_3, B'_4$  from $B_3, B_4$ by $m$th paratopy.

\hskip 30mm Check the existence of transversals in the layer-latin
cube with layers $A_0, A_1, B_2, B'_3, B'_4$.

\hskip 30mm  $m:=m+1$.

\hskip 25mm  $k:=k+1$.

\hskip 20mm $j:=j+1$

\hskip 10mm $i:=i+1$

\bigskip

Using Algorithm 1 and Algorithm 2 independently, we list all
layer-latin cubes of order $5$ with no transversals (see Appendix
1). In particular, we establish the following statement.

\begin{proposition}\label{PPVprop1120}
There are $123$ paratopy classes and $241$  incomplete paratopy
classes of $2$-layer latin cuboids of order $5$.
\end{proposition}

 We find out that every layer-latin cube of order $5$ with no
transversals satisfies the conditions of Lemmas \ref{PPVlemma1} or
\ref{PPVlemma2} up to equivalence. Consequently, Theorem
\ref{latinT1} is true.

\section{Nonextendible latin cuboids}

Recall that a latin cuboid $C$ of size $k\times q\times q$, $k<q$,
is called extendible (completable) if there exists a latin cuboid
$(k+1)\times q\times q$ (a latin cube of order $q$) containing $C$.
By definition, every completable  latin cuboid is extendible. But a
noncompletable   latin cuboid may be extendable. It is well known
that by the K\"onig--Hall theorem any latin rectangle of size
$k\times q$ is completable.
 The minimal example of nonextendable
latin cuboid of size $k\times q\times q$ was found by Kochol
\cite{Kochol}. It has size $2\times 5\times 5$ and  is provided in
\cite{MW}.

The following property is  known and  can be verified
straightforwardly.

\begin{proposition}\label{PPVprop100}\quad

$(1)$ Any latin cuboid of size $(q-1)\times q\times q$ or $1\times
q\times q$ is completable to a latin cube of order $q$.

$(2)$ Equivalent latin cuboids are extendible (completable) or
nonextendible (noncompletable) simultaneously.
\end{proposition}

By Proposition \ref{PPVprop100},  noncomplitable latin cuboids of
size $1\times 5\times 5$ and $4\times 5\times 5$ do not exist.
Moreover, all extendible latin cuboids of size $3\times 5\times 5$
are completable.

 Using the
list of paratopy classes of $2$-layer latin cuboids of order $5$, we
verify that

$(I)$ There is a unique paratopy class of nonextendible latin
cuboids of size $2\times 5\times 5$. The representative of this
class is

$$
   \begin{array}{ccccccccccc}
0 &1 &2 &3 &4 &\qquad & 1 &0 &4 &2 &3 \\
1 &0 &3 &4 &2 &\qquad & 0 &2 &1 &3 &4 \\
2 &3 &4 &0 &1 &\qquad & 3 &4 &0 &1 &2 \\
3 &4 &1 &2 &0 &\qquad & 4 &3 &2 &0 &1 \\
4 &2 &0 &1 &3 &\qquad & 2 &1 &3 &4 &0
   \end{array}
$$

$(II)$ There are $10$ paratopy classes of extendible but
noncomletable latin cuboids of size $2\times 5\times 5$. The list of
representatives of these paratopy classes presented in the Appendix
2.

$(III)$ There are $69$ paratopy classes of nonextendible latin
cuboids of size $3\times 5\times 5$. The list of representatives of
these paratopy classes can be found in the internet data archive
https://zenodo.org/records/10204026 .

$25$ out of  $69$ paratopy classes  contain $2\times 5\times 5$
cuboids from $10$ noncompletable paratopy classes, and $24$ out of
$69$ paratopy classes do not contain such subcuboids.

\newpage

\section{Appendix 1}

We find $123$ paratopy classes of layer-latin cubes of order $5$
with no transversals. $34$ classes among them can be represented  in
the form of $L[\pi_1,2],\dots, L[\pi_5,2]$ (see Corollary
\ref{corol_pi}), $67$ classes can be represented in the form of
$L_0,L[\varepsilon_{i_1},2],\dots, L[\varepsilon_{i_{4}},2]$ (see
Lemma \ref{PPVlemma2}), and $22$ classes can be represented in the
both ways. Below, the last $22$ classes are represented in the first
way. Thus, we  represent $56$ classes in the first way and the
remaining $67$ classes in the second way.

The first $56 $ classes of layer-latin cubes of order $5$ with no
transversals    can be divided  into three sets.

The first $30$ classes consist of latin squares
$L[\varepsilon_0,2]$, $L[\varepsilon_0,2]$, $L[\varepsilon_1,2]$,
and two latin squares $L[\pi_4,2]$ and $L[\pi_5,2]$ obtained from
$L[\varepsilon_0,2]$ by permutations of rows and columns. In Table
\ref{tab001}, each permutation is given by the first column of the
corresponding latin square. For example, the first row of the table
corresponds to the five squares
%\qquad
$$
   \begin{array}{ccccccccccccccccc}
  0 &1 &2 &3 &4 &\qquad & 0 &1 &2 &3 &4 &\qquad & 1 &2 &3 &4 &0 \\
  1 &2 &3 &4 &0 &\qquad & 1 &2 &3 &4 &0 &\qquad & 2 &3 &4 &0 &1 \\
  2 &3 &4 &0 &1 &\qquad & 2 &3 &4 &0 &1 &\qquad & 3 &4 &0 &1 &2 \\
  3 &4 &0 &1 &2 &\qquad & 3 &4 &0 &1 &2 &\qquad & 4 &0 &1 &2 &3 \\
  4 &0 &1 &2 &3 &\qquad & 4 &0 &1 &2 &3 &\qquad & 0 &1 &2 &3 &4
   \end{array}
$$
$$
   \begin{array}{ccccccccccc}
0 &1 &2 &3 &4 &\qquad & 1 &2 &3 &4 &0 \\
1 &2 &3 &4 &0 &\qquad & 2 &3 &4 &0 &1 \\
2 &3 &4 &0 &1 &\qquad & 3 &4 &0 &1 &2 \\
4 &0 &1 &2 &3 &\qquad & 0 &1 &2 &3 &4 \\
3 &4 &0 &1 &2 &\qquad & 4 &0 &1 &2 &3
   \end{array}
$$

\begin{table}
        \begin{center}

            \caption{Case $\varepsilon_0$, $\varepsilon_0$, $\varepsilon_1$.}\label{tab001}

\begin{tabular}{c|c} \hline
 $\pi_4$ & $\pi_5$  \\ \hline
 $0\ 1\ 2\ 4\ 3$ & $1\ 2\ 3\ 0\ 4 $ \\ \hline
 $0\ 1\ 2\ 4\ 3$ & $2\ 3\ 4\ 1\ 0  $ \\ \hline
 $0\ 1\ 3\ 4\ 2$ & $2\ 1\ 4\ 3\ 0  $ \\ \hline
 $0\ 1\ 4\ 2\ 3$ & $3\ 2\ 4\ 1\ 0  $ \\ \hline
 $0\ 1\ 4\ 3\ 2$ & $0\ 1\ 4\ 3\ 2  $ \\ \hline
 $0\ 1\ 4\ 3\ 2$ & $3\ 4\ 2\ 1\ 0  $ \\ \hline
 $0\ 2\ 1\ 4\ 3$ & $4\ 2\ 3\ 0\ 1  $ \\ \hline
 $0\ 2\ 3\ 4\ 1$ & $0\ 2\ 3\ 4\ 1  $ \\ \hline
 $0\ 2\ 3\ 4\ 1$ & $0\ 2\ 3\ 4\ 1  $ \\ \hline
 $0\ 3\ 1\ 4\ 2$ & $4\ 3\ 2\ 1\ 0  $ \\ \hline
 $0\ 3\ 4\ 1\ 2$ & $4\ 1\ 0\ 3\ 2  $ \\ \hline
 $0\ 3\ 4\ 2\ 1$ & $4\ 2\ 3\ 1\ 0  $ \\ \hline
 $0\ 3\ 4\ 2\ 1$ & $1\ 4\ 0\ 3\ 2  $ \\ \hline
 $0\ 4\ 1\ 2\ 3$ & $3\ 2\ 4\ 0\ 1  $ \\ \hline
 $0\ 4\ 1\ 2\ 3$ & $4\ 3\ 0\ 1\ 2  $ \\ \hline
 $1\ 0\ 3\ 4\ 2$ & $4\ 3\ 1\ 2\ 0  $ \\ \hline
 $1\ 0\ 4\ 2\ 3$ & $1\ 0\ 4\ 2\ 3  $ \\ \hline
 $1\ 4\ 0\ 2\ 3$ & $3\ 0\ 4\ 2\ 1  $ \\ \hline
 $2\ 4\ 0\ 1\ 3$ & $2\ 4\ 0\ 1\ 3  $ \\ \hline  \hline
 $2\ 3\ 4\ 0\ 1$ & $1\ 4\ 2\ 3\ 0 $ \\ \hline
 $2\ 3\ 4\ 0\ 1$ & $1\ 4\ 3\ 2\ 0 $ \\ \hline
 $2\ 3\ 4\ 0\ 1$ & $3\ 4\ 1\ 2\ 0 $ \\ \hline
 $2\ 3\ 4\ 0\ 1$ & $3\ 4\ 2\ 1\ 0 $ \\ \hline
 $2\ 3\ 4\ 0\ 1$ & $4\ 1\ 2\ 3\ 0 $ \\ \hline
 $2\ 3\ 4\ 0\ 1$ & $4\ 1\ 3\ 2\ 0 $ \\ \hline
 $3\ 4\ 0\ 1\ 2$ & $2\ 1\ 4\ 0\ 3 $ \\ \hline
 $3\ 4\ 0\ 1\ 2$ & $4\ 1\ 2\ 0\ 3 $ \\ \hline
 $4\ 0\ 1\ 2\ 3$ & $1\ 2\ 0\ 4\ 3 $ \\ \hline
 $4\ 0\ 1\ 2\ 3$ & $1\ 2\ 4\ 0\ 3 $ \\ \hline
 $4\ 0\ 1\ 2\ 3$ & $1\ 3\ 0\ 4\ 2 $ \\ \hline
 $4\ 0\ 1\ 2\ 3$ & $1\ 3\ 4\ 0\ 2 $ \\ \hline
\end{tabular}

        \end{center}
    \end{table}

\begin{table}
        \begin{center}

            \caption{Case $\varepsilon_0$, $\varepsilon_0$, $\varepsilon_0$.}\label{tab000}

\begin{tabular}{c|c} \hline
 $\pi_4$ & $\pi_5$  \\ \hline
 $0\ 1\ 2\ 4\ 3$ & $2\ 3\ 4\ 1\ 0 $ \\ \hline
 $0\ 1\ 3\ 4\ 2$ & $3\ 2\ 0\ 4\ 1 $ \\ \hline
 $0\ 2\ 1\ 4\ 3$ & $0\ 3\ 4\ 1\ 2 $ \\ \hline
 $0\ 2\ 3\ 4\ 1$ & $0\ 2\ 3\ 4\ 1 $ \\ \hline
 $0\ 2\ 3\ 4\ 1$ & $1\ 3\ 4\ 0\ 2 $ \\ \hline
 $0\ 3\ 1\ 4\ 2$ & $0\ 4\ 3\ 2\ 1 $ \\ \hline \hline
 $2\ 3\ 4\ 0\ 1$ & $1\ 2\ 4\ 3\ 0 $ \\ \hline
 $2\ 3\ 4\ 0\ 1$ & $1\ 3\ 4\ 2\ 0 $ \\ \hline
 $2\ 3\ 4\ 0\ 1$ & $2\ 1\ 4\ 3\ 0 $ \\ \hline
 $2\ 3\ 4\ 0\ 1$ & $4\ 1\ 2\ 3\ 0 $ \\ \hline
 $2\ 3\ 4\ 0\ 1$ & $4\ 1\ 3\ 2\ 0 $ \\ \hline
 $2\ 3\ 4\ 0\ 1$ & $4\ 3\ 1\ 2\ 0 $ \\ \hline
 $3\ 4\ 0\ 1\ 2$ & $3\ 4\ 2\ 1\ 0 $ \\ \hline
 $3\ 4\ 0\ 1\ 2$ & $1\ 4\ 2\ 3\ 0 $ \\ \hline
\end{tabular}

        \end{center}
    \end{table}

\begin{table}
        \begin{center}

            \caption{Case $\varepsilon_0$, $\varepsilon_1$, $\varepsilon_2$.}\label{tab012}

\begin{tabular}{c|c} \hline
 $\pi_4$ & $\pi_5$  \\ \hline
 $0\ 1\ 2\ 4\ 3$ & $0\ 1\ 2\ 4\ 3 $ \\ \hline
 $0\ 1\ 2\ 4\ 3$ & $4\ 0\ 1\ 3\ 2 $ \\ \hline
 $0\ 1\ 3\ 4\ 2$ & $0\ 4\ 2\ 1\ 3 $ \\ \hline
 $0\ 1\ 4\ 2\ 3$ & $1\ 0\ 2\ 4\ 3 $ \\ \hline
 $0\ 1\ 4\ 3\ 2$ & $1\ 2\ 0\ 4\ 3 $ \\ \hline
 $0\ 1\ 4\ 3\ 2$ & $3\ 4\ 2\ 1\ 0 $ \\ \hline
 $0\ 2\ 3\ 4\ 1$ & $2\ 4\ 0\ 1\ 3 $ \\ \hline
 $0\ 2\ 4\ 3\ 1$ & $4\ 2\ 0\ 1\ 3 $ \\ \hline
 $0\ 3\ 1\ 4\ 2$ & $2\ 1\ 0\ 4\ 3 $ \\ \hline
 $0\ 4\ 3\ 1\ 2$ & $0\ 4\ 3\ 1\ 2 $ \\ \hline \hline
 $3\ 4\ 0\ 1\ 2$ & $1\ 2\ 3\ 0\ 4 $ \\ \hline
 $3\ 4\ 0\ 1\ 2$ & $1\ 3\ 2\ 0\ 4 $ \\ \hline
\end{tabular}

        \end{center}
    \end{table}

Other $14$ layer-latin cubes with no transversals consist of
$L[\varepsilon_0,2]$, $L[\varepsilon_0,2]$, $L[\varepsilon_0,2]$,
and $L[\pi_4,2]$, $L[\pi_5,2]$ layers. Two last  latin squares
$L[\pi_4,2]$ and $L[\pi_5,2]$ are represented in Table \ref{tab000}
by permutations of rows.

The remaining $12$ layer-latin cubes with no transversals consist of
$L[\varepsilon_0,2]$, $L[\varepsilon_1,2]$, $L[\varepsilon_2,2]$,
and  $L[\pi_4,2]$, $L[\pi_5,2]$ layers. Two last  latin squares
$L[\pi_4,2]$ and  $L[\pi_5,2]$ are represented in Table
\ref{tab012}.

 Tables \ref{tab001}, \ref{tab000}, \ref{tab012} are divided into
two parts. In the lower part there are $22$ classes, which can be
represented by the second way.

The remaining $67$ classes of layer-latin cubes of order $5$ with no
transversals can be  presented as $L_0,L[\varepsilon_{i_1},2],\dots,
L[\varepsilon_{i_{4}},2]$ (see Lemma \ref{PPVlemma2}). We divide
these classes into sets by the values $i_1,\dots, i_{4}$.

When $i_1=i_2=i_3=i_4=0$, we have $3$ classes obtained by adding
$L_0$ from the following set of nonlinear squares
$$
   \begin{array}{ccccccccccccccccc}
  1 &0 &3 &4 &2 &\qquad & 1 &0 &3 &4 &2 &\qquad & 1 &0 &3 &4 &2 \\
  0 &3 &1 &2 &4 &\qquad & 0 &4 &1 &2 &3 &\qquad & 0 &4 &1 &2 &3 \\
  4 &2 &0 &1 &3 &\qquad & 3 &2 &0 &1 &4 &\qquad & 3 &2 &0 &1 &4 \\
  2 &1 &4 &3 &0 &\qquad & 4 &1 &2 &3 &0 &\qquad & 4 &3 &2 &0 &1 \\
  3 &4 &2 &0 &1 &\qquad & 2 &3 &4 &0 &1 &\qquad & 2 &1 &4 &3 &0
   \end{array}
$$
and $4$ classes obtained by adding $L_0$ from the following set of
linear squares
$$
   \begin{array}{ccccccccccccccccccccccc}
  1 &0 &3 &4 &2 &\qquad & 1 &2 &3 &4 &0 &\qquad & 1 &2 &4 &0 &3 &\qquad & 1 &2 &4 &0 &3 \\
  0 &3 &4 &2 &1 &\qquad & 2 &3 &4 &0 &1 &\qquad & 2 &0 &1 &3 &4 &\qquad & 2 &3 &0 &1 &4 \\
  3 &4 &2 &1 &0 &\qquad & 3 &4 &0 &1 &2 &\qquad & 3 &4 &0 &1 &2 &\qquad & 3 &4 &1 &2 &0 \\
  4 &2 &1 &0 &3 &\qquad & 4 &0 &1 &2 &3 &\qquad & 4 &1 &3 &2 &0 &\qquad & 4 &0 &2 &3 &1 \\
  2 &1 &0 &3 &4 &\qquad & 0 &1 &2 &3 &4 &\qquad & 0 &3 &2 &4 &1 &\qquad & 0 &1 &3 &4 &2
   \end{array}
$$

When $i_1=i_2=i_3=0, \ i_4=1$, we have $12$ classes obtained by
adding $L_0$ from the following set of nonlinear squares
$$
   \begin{array}{ccccccccccccccccc}
  0 &1 &2 &3 &4 &\qquad & 0 &1 &3 &4 &2 &\qquad & 0 &1 &3 &4 &2 \\
  1 &2 &4 &0 &3 &\qquad & 2 &0 &4 &1 &3 &\qquad & 2 &3 &4 &1 &0 \\
  3 &4 &0 &2 &1 &\qquad & 3 &4 &2 &0 &1 &\qquad & 3 &4 &0 &2 &1 \\
  4 &0 &3 &1 &2 &\qquad & 1 &2 &0 &3 &4 &\qquad & 1 &0 &2 &3 &4 \\
  2 &3 &1 &4 &0 &\qquad & 4 &3 &1 &2 &0 &\qquad & 4 &2 &1 &0 &3
   \end{array}
$$
$$
   \begin{array}{ccccccccccccccccc}
  0 &1 &3 &4 &2 &\qquad & 0 &1 &3 &4 &2 &\qquad & 0 &1 &3 &4 &2 \\
  2 &4 &0 &1 &3 &\qquad & 2 &4 &0 &1 &3 &\qquad & 2 &4 &1 &0 &3 \\
  3 &0 &4 &2 &1 &\qquad & 4 &3 &2 &0 &1 &\qquad & 4 &0 &2 &3 &1 \\
  4 &2 &1 &3 &0 &\qquad & 3 &0 &1 &2 &4 &\qquad & 3 &2 &0 &1 &4 \\
  1 &3 &2 &0 &4 &\qquad & 1 &2 &4 &3 &0 &\qquad & 1 &3 &4 &2 &0
   \end{array}
$$
$$
   \begin{array}{ccccccccccccccccc}
  0 &1 &3 &4 &2 &\qquad & 0 &1 &4 &3 &2 &\qquad & 0 &1 &4 &3 &2 \\
  3 &4 &1 &2 &0 &\qquad & 2 &0 &1 &4 &3 &\qquad & 2 &0 &1 &4 &3 \\
  4 &0 &2 &1 &3 &\qquad & 3 &4 &0 &2 &1 &\qquad & 4 &3 &0 &2 &1 \\
  1 &2 &0 &3 &4 &\qquad & 4 &2 &3 &1 &0 &\qquad & 3 &4 &2 &1 &0 \\
  2 &3 &4 &0 &1 &\qquad & 1 &3 &2 &0 &4 &\qquad & 1 &2 &3 &0 &4
   \end{array}
$$
$$
   \begin{array}{ccccccccccccccccc}
  0 &1 &4 &3 &2 &\qquad & 0 &2 &3 &4 &1 &\qquad & 2 &3 &0 &4 &1 \\
  2 &3 &0 &4 &1 &\qquad & 3 &0 &4 &1 &2 &\qquad & 1 &0 &4 &2 &3 \\
  4 &0 &1 &2 &3 &\qquad & 4 &1 &0 &2 &3 &\qquad & 3 &4 &1 &0 &2 \\
  3 &4 &2 &1 &0 &\qquad & 1 &4 &2 &3 &0 &\qquad & 4 &2 &3 &1 &0 \\
  1 &2 &3 &0 &4 &\qquad & 2 &3 &1 &0 &4 &\qquad & 0 &1 &2 &3 &4
   \end{array}
$$
and $6$ classes obtained by adding $L_0$ from the set of following
linear squares
$$
   \begin{array}{ccccccccccccccccc}
  0 &1 &2 &3 &4 &\qquad & 0 &1 &3 &4 &2 &\qquad & 0 &2 &3 &1 &4 \\
  1 &3 &0 &4 &2 &\qquad & 1 &4 &0 &2 &3 &\qquad & 1 &3 &0 &4 &2 \\
  2 &0 &4 &1 &3 &\qquad & 3 &0 &2 &1 &4 &\qquad & 3 &4 &2 &0 &1 \\
  3 &4 &1 &2 &0 &\qquad & 4 &2 &1 &3 &0 &\qquad & 4 &0 &1 &2 &3 \\
  4 &2 &3 &0 &1 &\qquad & 2 &3 &4 &0 &1 &\qquad & 2 &1 &4 &3 &0
   \end{array}
$$
$$
   \begin{array}{ccccccccccccccccc}
  0 &3 &4 &1 &2 &\qquad & 1 &2 &3 &4 &0 &\qquad & 2 &3 &4 &0 &1 \\
  2 &4 &1 &0 &3 &\qquad & 2 &3 &4 &0 &1 &\qquad & 3 &4 &0 &1 &2 \\
  4 &0 &2 &3 &1 &\qquad & 3 &4 &0 &1 &2 &\qquad & 4 &0 &1 &2 &3 \\
  3 &1 &0 &2 &4 &\qquad & 4 &0 &1 &2 &3 &\qquad & 0 &1 &2 &3 &4 \\
  1 &2 &3 &4 &0 &\qquad & 0 &1 &2 &3 &4 &\qquad & 1 &2 &3 &4 &0
   \end{array}
$$

When $i_1=i_2=0, \ i_3=i_4=1$, we have $6$ classes obtained by
adding $L_0$ from the following set of nonlinear squares
$$
   \begin{array}{ccccccccccccccccc}
  0 &1 &2 &3 &4 &\qquad & 0 &1 &2 &4 &3 &\qquad & 0 &1 &2 &4 &3 \\
  1 &3 &0 &4 &2 &\qquad & 1 &2 &0 &3 &4 &\qquad & 1 &2 &3 &0 &4 \\
  2 &4 &1 &0 &3 &\qquad & 3 &4 &1 &2 &0 &\qquad & 4 &3 &0 &1 &2 \\
  3 &0 &4 &2 &1 &\qquad & 2 &3 &4 &0 &1 &\qquad & 2 &0 &4 &3 &1 \\
  4 &2 &3 &1 &0 &\qquad & 4 &0 &3 &1 &2 &\qquad & 3 &4 &1 &2 &0
   \end{array}
$$
$$
   \begin{array}{ccccccccccccccccc}
  0 &1 &2 &4 &3 &\qquad & 0 &1 &2 &4 &3 &\qquad & 0 &1 &2 &4 &3 \\
  1 &4 &0 &3 &2 &\qquad & 1 &4 &3 &0 &2 &\qquad & 2 &3 &0 &1 &4 \\
  3 &0 &4 &2 &1 &\qquad & 4 &3 &1 &2 &0 &\qquad & 1 &4 &3 &0 &2 \\
  2 &3 &1 &0 &4 &\qquad & 2 &0 &4 &3 &1 &\qquad & 3 &0 &4 &2 &1 \\
  4 &2 &3 &1 &0 &\qquad & 3 &2 &0 &1 &4 &\qquad & 4 &2 &1 &3 &0
   \end{array}
$$
and $7$ classes obtained by adding $L_0$ from the following set of
linear squares
$$
   \begin{array}{ccccccccccccccccccccccc}
  0 &1 &2 &4 &3 &\qquad & 0 &1 &4 &2 &3 &\qquad & 0 &1 &4 &2 &3 &\qquad & 0 &2 &1 &4 &3 \\
  1 &2 &3 &0 &4 &\qquad & 1 &2 &0 &3 &4 &\qquad & 1 &3 &2 &0 &4 &\qquad & 1 &3 &2 &0 &4 \\
  2 &3 &4 &1 &0 &\qquad & 2 &3 &1 &4 &0 &\qquad & 2 &0 &3 &4 &1 &\qquad & 2 &4 &3 &1 &0 \\
  3 &4 &0 &2 &1 &\qquad & 3 &4 &2 &0 &1 &\qquad & 3 &4 &0 &1 &2 &\qquad & 3 &0 &4 &2 &1 \\
  4 &0 &1 &3 &2 &\qquad & 4 &0 &3 &1 &2 &\qquad & 4 &2 &1 &3 &0 &\qquad & 4 &1 &0 &3 &2
   \end{array}
$$
$$
   \begin{array}{ccccccccccccccccc}
  0 &2 &1 &4 &3 &\qquad & 0 &3 &4 &2 &1 &\qquad & 2 &3 &4 &0 &1 \\
  1 &4 &2 &3 &0 &\qquad & 1 &4 &0 &3 &2 &\qquad & 3 &4 &0 &1 &2 \\
  4 &0 &3 &1 &2 &\qquad & 2 &0 &1 &4 &3 &\qquad & 4 &0 &1 &2 &3 \\
  2 &3 &4 &0 &1 &\qquad & 3 &1 &2 &0 &4 &\qquad & 0 &1 &2 &3 &4 \\
  3 &1 &0 &2 &4 &\qquad & 4 &2 &3 &1 &0 &\qquad & 1 &2 &3 &4 &0
   \end{array}
$$

When $i_1=i_2=0, \ i_3=1, \ i_4=2$, we have $12$ classes obtained by
adding $L_0$ from the following set of nonlinear squares
$$
   \begin{array}{ccccccccccccccccc}
  0 &1 &2 &4 &3 &\qquad & 0 &1 &2 &4 &3 &\qquad & 0 &1 &2 &4 &3 \\
  1 &0 &3 &2 &4 &\qquad & 1 &2 &3 &0 &4 &\qquad & 2 &0 &1 &3 &4 \\
  3 &2 &4 &1 &0 &\qquad & 2 &4 &0 &3 &1 &\qquad & 1 &3 &4 &0 &2 \\
  4 &3 &1 &0 &2 &\qquad & 4 &3 &1 &2 &0 &\qquad & 4 &2 &3 &1 &0 \\
  2 &4 &0 &3 &1 &\qquad & 3 &0 &4 &1 &2 &\qquad & 3 &4 &0 &2 &1
   \end{array}
$$
$$
   \begin{array}{ccccccccccccccccc}
  0 &1 &2 &4 &3 &\qquad & 0 &1 &2 &4 &3 &\qquad & 0 &1 &2 &4 &3 \\
  2 &0 &4 &3 &1 &\qquad & 2 &0 &4 &3 &1 &\qquad & 4 &0 &3 &2 &1 \\
  1 &2 &3 &0 &4 &\qquad & 3 &2 &0 &1 &4 &\qquad & 1 &2 &0 &3 &4 \\
  4 &3 &1 &2 &0 &\qquad & 1 &4 &3 &2 &0 &\qquad & 2 &3 &4 &1 &0 \\
  3 &4 &0 &1 &2 &\qquad & 4 &3 &1 &0 &2 &\qquad & 3 &4 &1 &0 &2
   \end{array}
$$
$$
   \begin{array}{ccccccccccccccccc}
  0 &1 &2 &4 &3 &\qquad & 0 &1 &3 &4 &2 &\qquad & 0 &1 &3 &4 &2 \\
  4 &0 &3 &2 &1 &\qquad & 2 &0 &1 &3 &4 &\qquad & 4 &0 &1 &2 &3 \\
  1 &4 &0 &3 &2 &\qquad & 3 &4 &2 &1 &0 &\qquad & 1 &2 &0 &3 &4 \\
  3 &2 &4 &1 &0 &\qquad & 1 &2 &4 &0 &3 &\qquad & 2 &3 &4 &1 &0 \\
  2 &3 &1 &0 &4 &\qquad & 4 &3 &0 &2 &1 &\qquad & 3 &4 &2 &0 &1
   \end{array}
$$
$$
   \begin{array}{ccccccccccccccccc}
  0 &2 &1 &4 &3 &\qquad & 1 &0 &3 &4 &2 &\qquad & 0 &1 &2 &3 &4 \\
  4 &0 &3 &2 &1 &\qquad & 4 &2 &1 &0 &3 &\qquad & 4 &2 &1 &0 &3 \\
  1 &4 &0 &3 &2 &\qquad & 2 &3 &4 &1 &0 &\qquad & 2 &3 &0 &4 &1 \\
  2 &3 &4 &1 &0 &\qquad & 3 &4 &0 &2 &1 &\qquad & 1 &4 &3 &2 &0 \\
  3 &1 &2 &0 &4 &\qquad & 0 &1 &2 &3 &4 &\qquad & 3 &0 &4 &1 &2
   \end{array}
$$
and $5$ classes obtained by adding $L_0$ from the following set of
linear squares
$$
   \begin{array}{ccccccccccccccccc}
  0 &1 &2 &3 &4 &\qquad & 0 &1 &3 &4 &2 &\qquad & 0 &2 &1 &4 &3 \\
  1 &2 &4 &0 &3 &\qquad & 1 &3 &2 &0 &4 &\qquad & 2 &1 &3 &0 &4 \\
  2 &4 &3 &1 &0 &\qquad & 3 &2 &4 &1 &0 &\qquad & 3 &4 &0 &1 &2 \\
  3 &0 &1 &4 &2 &\qquad & 4 &0 &1 &2 &3 &\qquad & 4 &0 &2 &3 &1 \\
  4 &3 &0 &2 &1 &\qquad & 2 &4 &0 &3 &1 &\qquad & 1 &3 &4 &2 &0
   \end{array}
$$
$$
   \begin{array}{ccccccccccc}
0 &4 &1 &3 &2 &\qquad & 3 &4 &0 &1 &2 \\
1 &0 &3 &2 &4 &\qquad & 4 &0 &1 &2 &3 \\
3 &1 &2 &4 &0 &\qquad & 0 &1 &2 &3 &4 \\
4 &2 &0 &1 &3 &\qquad & 1 &2 &3 &4 &0 \\
2 &3 &4 &0 &1 &\qquad & 2 &3 &4 &0 &1
   \end{array}
$$

When $i_1=i_2=0, \ i_3=1, \ i_4=4$, we have $6$ classes obtained by
adding $L_0$ from the following set of nonlinear squares
$$
   \begin{array}{ccccccccccccccccc}
  1 &0 &3 &4 &2 &\qquad & 1 &0 &3 &4 &2 &\qquad & 1 &0 &3 &4 &2 \\
  0 &3 &1 &2 &4 &\qquad & 0 &4 &1 &2 &3 &\qquad & 0 &4 &1 &2 &3 \\
  4 &2 &0 &1 &3 &\qquad & 3 &2 &0 &1 &4 &\qquad & 3 &2 &0 &1 &4 \\
  2 &1 &4 &3 &0 &\qquad & 4 &1 &2 &3 &0 &\qquad & 4 &3 &2 &0 &1 \\
  3 &4 &2 &0 &1 &\qquad & 2 &3 &4 &0 &1 &\qquad & 2 &1 &4 &3 &0
   \end{array}
$$
$$
   \begin{array}{ccccccccccccccccc}
  1 &0 &3 &4 &2 &\qquad & 1 &0 &3 &4 &2 &\qquad & 3 &4 &0 &2 &1 \\
  2 &3 &0 &1 &4 &\qquad & 2 &4 &0 &1 &3 &\qquad & 0 &1 &2 &3 &4 \\
  0 &4 &1 &2 &3 &\qquad & 4 &1 &2 &3 &0 &\qquad & 1 &0 &3 &4 &2 \\
  4 &1 &2 &3 &0 &\qquad & 0 &3 &1 &2 &4 &\qquad & 4 &2 &1 &0 &3 \\
  3 &2 &4 &0 &1 &\qquad & 3 &2 &4 &0 &1 &\qquad & 2 &3 &4 &1 &0
   \end{array}
$$
and $2$ classes obtained by adding $L_0$ from the following  set of
linear squares
$$
   \begin{array}{ccccccccccc}
1 &2 &4 &0 &3 &\qquad & 1 &3 &4 &2 &0 \\
2 &0 &1 &3 &4 &\qquad & 2 &4 &1 &0 &3 \\
3 &4 &0 &1 &2 &\qquad & 4 &0 &3 &1 &2 \\
4 &1 &3 &2 &0 &\qquad & 0 &1 &2 &3 &4 \\
0 &3 &2 &4 &1 &\qquad & 3 &2 &0 &4 &1
   \end{array}
$$

When $i_1=0, \ i_2=1, \ i_3=2, \ i_4=3$, we have $3$ classes
obtained by adding $L_0$ from the following set of nonlinear squares
$$
   \begin{array}{ccccccccccccccccc}
  0 &1 &2 &3 &4 &\qquad & 0 &3 &2 &1 &4 &\qquad & 0 &3 &2 &1 &4 \\
  2 &4 &0 &1 &3 &\qquad & 2 &0 &4 &3 &1 &\qquad & 2 &4 &0 &3 &1 \\
  4 &3 &1 &2 &0 &\qquad & 4 &1 &3 &2 &0 &\qquad & 1 &2 &4 &0 &3 \\
  3 &2 &4 &0 &1 &\qquad & 3 &4 &1 &0 &2 &\qquad & 3 &0 &1 &4 &2 \\
  1 &0 &3 &4 &2 &\qquad & 1 &2 &0 &4 &3 &\qquad & 4 &1 &3 &2 &0
   \end{array}
$$
and $1$ class with the linear square $L_0=
   \begin{array}{ccccc}
0 &1 &2 &3 &4  \\
2 &0 &4 &1 &3  \\
4 &2 &3 &0 &1  \\
1 &3 &0 &4 &2  \\
3 &4 &1 &2 &0
   \end{array}
$

%\newpage

\section{Appendix 2}

The list of $10$ representatives of paratopy classes of extendible
but noncomletable latin cuboids of size $2 \times 5 \times 5$.

$$
   \begin{array}{cccccccccccc}
0 &1 &2 &3 &4 &\qquad & 1 &0 &3 &4 &2 & \\
1 &0 &3 &4 &2 &\qquad & 2 &3 &1 &0 &4 & \\
2 &3 &4 &0 &1 &\qquad & 0 &4 &2 &1 &3 & \qquad \text{may be completed to a latin cuboid} \\
3 &4 &1 &2 &0 &\qquad & 4 &2 &0 &3 &1 & \qquad \text{of size $3 \times 5 \times 5$ by $2$ ways,}\\
4 &2 &0 &1 &3 &\qquad & 3 &1 &4 &2 &0 &
   \end{array}
$$
$$
   \begin{array}{cccccccccccc}
0 &1 &2 &3 &4 &\qquad & 1 &0 &4 &2 &3 & \\
1 &0 &3 &4 &2 &\qquad & 2 &3 &0 &1 &4 & \\
2 &3 &4 &0 &1 &\qquad & 0 &4 &1 &3 &2 & \qquad \text{ may be completed to a latin cuboid} \\
3 &4 &1 &2 &0 &\qquad & 4 &2 &3 &0 &1 & \qquad \text{of size $3 \times 5 \times 5$ by 6 ways,}\\
4 &2 &0 &1 &3 &\qquad & 3 &1 &2 &4 &0 &
   \end{array}
$$
$$
   \begin{array}{cccccccccccc}
0 &1 &2 &3 &4 &\qquad & 1 &0 &4 &2 &3 & \\
1 &0 &3 &4 &2 &\qquad & 2 &3 &0 &1 &4 & \\
2 &3 &4 &0 &1 &\qquad & 3 &1 &2 &4 &0 & \qquad \text{may be completed  to a latin cuboid} \\
3 &4 &1 &2 &0 &\qquad & 4 &2 &3 &0 &1 & \qquad \text{of size $3 \times 5 \times 5$ by 6 ways,}\\
4 &2 &0 &1 &3 &\qquad & 0 &4 &1 &3 &2 &
   \end{array}
$$
$$
   \begin{array}{cccccccccccc}
0 &1 &2 &3 &4 &\qquad & 1 &0 &4 &2 &3 & \\
1 &0 &3 &4 &2 &\qquad & 2 &3 &0 &1 &4 & \\
2 &3 &4 &0 &1 &\qquad & 3 &2 &1 &4 &0 & \qquad \text{may be completed  to a latin cuboid} \\
3 &4 &1 &2 &0 &\qquad & 4 &1 &3 &0 &2 & \qquad \text{of size $3 \times 5 \times 5$ by 4 ways,}\\
4 &2 &0 &1 &3 &\qquad & 0 &4 &2 &3 &1 &
   \end{array}
$$
$$
   \begin{array}{cccccccccccc}
0 &1 &2 &3 &4 &\qquad & 1 &0 &4 &2 &3 & \\
1 &0 &3 &4 &2 &\qquad & 2 &3 &1 &0 &4 & \\
2 &3 &4 &0 &1 &\qquad & 3 &4 &2 &1 &0 & \qquad \text{may be completed  to a latin cuboid} \\
3 &4 &1 &2 &0 &\qquad & 4 &2 &0 &3 &1 & \qquad \text{of size $3 \times 5 \times 5$ by 2 ways,}\\
4 &2 &0 &1 &3 &\qquad & 0 &1 &3 &4 &2 &
   \end{array}
$$
$$
   \begin{array}{cccccccccccc}
0 &1 &2 &3 &4 &\qquad & 1 &2 &0 &4 &3 & \\
1 &0 &3 &4 &2 &\qquad & 0 &3 &2 &1 &4 & \\
2 &3 &4 &0 &1 &\qquad & 4 &1 &3 &2 &0 & \qquad \text{may be completed  to a latin cuboid} \\
3 &4 &1 &2 &0 &\qquad & 2 &0 &4 &3 &1 & \qquad \text{of size $3 \times 5 \times 5$ by 6 ways,}\\
4 &2 &0 &1 &3 &\qquad & 3 &4 &1 &0 &2 &
   \end{array}
$$
$$
   \begin{array}{cccccccccccc}
0 &1 &2 &3 &4 &\qquad & 1 &2 &0 &4 &3 & \\
1 &0 &3 &4 &2 &\qquad & 0 &3 &4 &2 &1 & \\
2 &3 &4 &0 &1 &\qquad & 3 &0 &2 &1 &4 & \qquad \text{may be completed  to a latin cuboid} \\
3 &4 &1 &2 &0 &\qquad & 4 &1 &3 &0 &2 & \qquad \text{of size $3 \times 5 \times 5$ by 6 ways,}\\
4 &2 &0 &1 &3 &\qquad & 2 &4 &1 &3 &0 &
   \end{array}
$$
$$
   \begin{array}{cccccccccccc}
0 &1 &2 &3 &4 &\qquad & 1 &2 &0 &4 &3 & \\
1 &0 &3 &4 &2 &\qquad & 2 &3 &4 &1 &0 & \\
2 &3 &4 &0 &1 &\qquad & 0 &1 &3 &2 &4 & \qquad \text{may be completed  to a latin cuboid} \\
3 &4 &1 &2 &0 &\qquad & 4 &2 &2 &3 &1 & \qquad \text{of size $3 \times 5 \times 5$ by 12 ways,}\\
4 &2 &0 &1 &3 &\qquad & 3 &4 &1 &0 &2 &
   \end{array}
$$
$$
   \begin{array}{cccccccccccc}
0 &1 &2 &3 &4 &\qquad & 1 &2 &0 &4 &3 & \\
1 &0 &3 &4 &2 &\qquad & 2 &3 &4 &1 &0 & \\
2 &3 &4 &0 &1 &\qquad & 3 &0 &1 &2 &4 & \qquad \text{may be completed  to a latin cuboid} \\
3 &4 &1 &2 &0 &\qquad & 4 &1 &3 &0 &2 & \qquad \text{of size $3 \times 5 \times 5$ by 3 ways,}\\
4 &2 &0 &1 &3 &\qquad & 0 &4 &2 &3 &1 &
   \end{array}
$$
$$
   \begin{array}{cccccccccccc}
0 &1 &2 &3 &4 &\qquad & 1 &2 &0 &4 &3 & \\
1 &0 &3 &4 &2 &\qquad & 4 &3 &2 &1 &0 & \\
2 &3 &4 &0 &1 &\qquad & 0 &1 &3 &2 &4 & \qquad \text{may be completed  to a latin cuboid} \\
3 &4 &1 &2 &0 &\qquad & 2 &0 &4 &3 &1 & \qquad \text{of size $3 \times 5 \times 5$ by 2 ways.}\\
4 &2 &0 &1 &3 &\qquad & 3 &4 &1 &0 &2 &
   \end{array}
$$

\newpage

%Note that there exists a noncompletable $5\times5\times2$  latin cuboid  \cite{McKW}.

\end{document}